\documentclass[12pt,a4paper]{article}
\usepackage{euscript,amsfonts,amssymb,amsmath,amscd,amsthm,enumerate}

\usepackage{graphicx}

\usepackage{color}

\newtheorem{theorem}{Theorem} 
\newtheorem*{theorem*}{Theorem}

\newtheorem{proposition}[theorem]{Proposition}

\theoremstyle{remark}
\newtheorem{rmk}[theorem]{Remark}

\newcommand{\rr}{\mathbb{R}}
\newcommand{\zz}{\mathbb{Z}}
\newcommand{\nn}{\mathbb{N}}
\newcommand{\ep}{\hfill \ensuremath{\Box}}

\title{Limits of Multilevel TASEP and similar processes}

\author{Vadim Gorin\thanks{Institute for Information Transmission Problems of Russian Academy of Sciences,
e-mail: vadicgor@gmail.com} \and Mykhaylo Shkolnikov\thanks{INTECH Investment Management, e-mail: mshkolni@gmail.com}}

\begin{document}

\maketitle

\begin{abstract}
 We study the asymptotic behavior of a class of stochastic dynamics on interlacing particle configurations (also known as Gelfand-Tsetlin patterns).    Examples of such dynamics include, in particular, a multi-layer extension of TASEP and particle dynamics related to the shuffling algorithm for
 domino tilings of the Aztec diamond. We prove that the process of reflected interlacing Brownian motions introduced by Warren in \cite{W} serves as a  universal scaling limit for such dynamics.
\end{abstract}

\section{Introduction}
Consider $N(N+1)/2$ interlacing particles with integer coordinates $x_i^j$, $j=1,\dots, N$,
$i=1,\dots,j$ satisfying the inequalities
\begin{equation}\label{int_ineq}
 x_{i-1}^j < x_{i-1}^{j-1} \le x_{i}^{j}.
\end{equation}
and denote by $\mathbb {GT}^{(N)}$ the set of all vectors in ${\mathbb Z}^{N(N+1)/2}$, which
satisfy \eqref{int_ineq}. An element of $\mathbb{GT}^{(5)}$ is shown in Figure
\ref{Figure_Interlacing}. Elements of $\mathbb {GT}^{(N)}$ are often called Gelfand--Tsetlin
patterns of size $N$ and under this name are widely used in representation-theoretic context.
Whenever it does not lead to confusion, we use the notation $x_i^j$ both for the location of a
particle and the particle itself.
\begin{figure}[h]
\begin{center}
\noindent{\scalebox{1.3}{\includegraphics{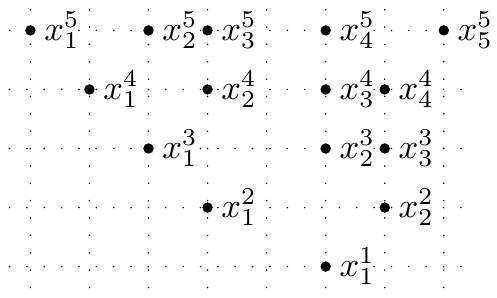}}} \caption{Interlacing particles.
\label{Figure_Interlacing} }
\end{center}
\end{figure}

\smallskip

In this article we study a class of Markov chains on $\mathbb{GT}^{(N)}$ with simple block/push
interaction between particles. Let us start with an example. Let $Y(t)$, $t\geq0$ be a
continuous-time Markov chain taking values in $\mathbb {GT}^{(N)}$ and defined as follows. Each of
the $N(N+1)/2$ particles has an independent exponential clock of rate $1$ (in other words, the times
when clocks of particles ring are independent standard Poisson processes). If the clock of the particle
$Y_i^j$ rings at time $t$, we check whether $Y_{i}^{j-1}(t-)=Y_{i}^j(t-)+1$. If so, then nothing
happens; otherwise we let $Y_i^j(t)=Y_i^j(t-)+1$. If
$Y_i^j(t-)=Y_{i+1}^{j+1}(t-)=\dots=Y_{i+k}^{j+k}(t-)$, then we also set
$Y_{i+1}^{j+1}(t)=Y_{i+1}^{j+1}(t-)+1,\dots,Y_{i+k}^{j+k}(t)=Y_{i+k}^{j+k}(t-)+1$.

\smallskip

Informally, one can think of each particle having  a weight depending on its vertical coordinate,
with higher particles being lighter. The dynamics is defined in such a way that heavier particles
push lighter ones and lighter ones are blocked by heavier ones in order to preserve the
interlacement conditions.

\smallskip

The Markov chain $Y$ was introduced in \cite{BF} as an example of a 2d growth model relating
classical interacting particle systems and random surfaces arising from dimers. The restriction of
$Y$ to the $N$ leftmost particles $x_1^1,\dots,x_1^N$ is the familiar totally asymmetric exclusion
process (TASEP), while the particle configuration $Y(t)$ at a fixed time $t$ can be identified
with a lozenge tiling of a sector in the plane and with a stepped surface (see the introduction in
\cite{BF} for the details).

\smallskip

More generally, we can start from any $N(N+1)/2$ integer-valued random processes (in discrete or
continuous time) with unit jumps of particles and construct a process similar to $Y$ on
$\mathbb{GT}^{(N)}$ through the same block/push interactions between particles (see Section
\ref{Section_definitions} for a formal definition). We will refer to this process as the
\emph{interlacing dynamics} driven by the original $N(N+1)/2$ random processes. Some of such
processes were studied in \cite{BF}, \cite{WW} and \cite{N}. In particular, in the last article it
is shown that interlacing dynamics driven by $N(N+1)/2$ independent simple random walks is closely
related to the shuffling algorithm for the domino tilings of the Aztec diamond studied in
\cite{EKLP}.

\smallskip

In the present article we investigate the convergence to the scaling limits of the process $Y$ and
more general interlacing dynamics. The asymptotics of the \emph{fixed time} distributions of such
processes as $t\to\infty$ is well studied. Most results are easier to prove if the process is
started from the \emph{densely packed} initial conditions: $Y_i^j(0)=1-i$, $j=1,\dots, N$,
$i=1,\dots, j$. The reason is that such processes often have \emph{determinantal} correlation
functions (see \cite{BF}). However, many of the results hold in much greater generality. If one
takes the limit $t\to\infty$ for a fixed $N\in\nn$, then the distribution of $Y(t)$ converges to
the so-called ``GUE-minors'' process, namely the joint distribution of the eigenvalues of an
$N\times N$ Gaussian Hermitian matrix and the eigenvalues of its top-left corners (see \cite{Bar},
\cite{JN}, \cite{OR}, \cite{N}).

\smallskip

If $N$ goes to infinity together with $t$, then  locally one unveils the Tracy-Widom distribution,
the Airy point process and discrete point processes governed by the sine kernel in the limit (see
\cite{BK}, \cite{J1}, \cite{J2}), whereas the global fluctuations in this model are described by
the Gaussian Free Field (see \cite{BF}). In the physics terminology the model fits into the
anisotropic KPZ universality class (see the introduction in \cite{BF} and the references therein
for more details).

\smallskip

We perform the next step and study the joint distributions of $Y$ and similar processes at
\emph{various} times. Currently we restrict ourselves to the case of fixed $N$. The limit process
$W_i^j$, $j=1,\dots,N$, $i=1,\dots,j$ was introduced in \cite{W} and can be constructed
inductively as follows. $W_1^1$ is standard Brownian motion; given $W_i^j$ with $j<k$ the
coordinate $W_i^k$ is defined as a Brownian motion \emph{reflected} by the trajectories
$W_{i-1}^{k-1}$ (if $i>1$) and $W_{i}^{k-1}$ (if $i<k$) (see \cite{W} or Section
\ref{Section_definitions} for a more detailed definition). The process $W$ has many interesting
properties: for instance, its projection $W_i^N$, $i=1,\dots,N$ is an $N$-dimensional Dyson's
Brownian motion, namely the process of $N$ independent Brownian motions conditioned not to collide
with each other by means of Doob's $h$-transform.

Our main result is the following theorem. Let $D([0,\infty),\mathbb R^{N(N+1)/2})$ stand for the
space of right-continuous functions from $[0,\infty)$ to $\mathbb R^{N(N+1)/2}$ having left
limits, endowed with the topology of uniform convergence on compact sets.

\begin{theorem}
\label{theorem_main} Let $X(\cdot;n)$, $n=1,2,\dots$ be a sequence of $\mathbb
Z^{N(N+1)/2}$-valued random processes with unit steps in each coordinate and let $\mathcal
X(\cdot;n)$, $n=1,2,\dots$ be the interlacing dynamics driven by $X(\cdot;n)$, $n\in\nn$,
respectively, with some initial conditions $\mathcal X(0;n)\in\mathbb {GT}^{(N)}$, $n\in\nn$,
which are independent of the increments of $X(\cdot;n)$, $n\in\nn$. Suppose that there exists
a sequence $\{a_n(\cdot)\}_{n\in\nn}$ of continuous real-valued functions on
$[0,\infty)$ and a sequence $\{b_n\}_{n\in\nn}$ of positive reals such that,
as $n\to\infty$, one has $b_n\to\infty$, the law of $(X(\cdot;n)-a_n(\cdot))/b_n$ on
$D([0,\infty),\mathbb R^{N(N+1)/2})$ converges to the law of a standard $N(N+1)/2$ dimensional
Brownian motion, and the law of $(\mathcal X(0;n)-a_n(0))/b_n$ converges to some law $W(0)$.

Then the law of $(\mathcal X(\cdot;n)-a_n(\cdot))/b_n$ on $D([0,\infty),\mathbb R^{N(N+1)/2})$ converges
to the law of the process $W$ with initial condition $W(0)$ as $n\to\infty$.
\end{theorem}

\smallskip

\begin{rmk}
If one wants to speak about discrete time processes, then it suffices to extrapolate them to
continuous time processes with piecewise constant trajectories, to which one can apply Theorem
\ref{theorem_main}.
\end{rmk}

\smallskip

\begin{rmk}
Theorem \ref{theorem_main} for a special choice of driving processes $X(\cdot;n)$ was conjectured
in \cite{N}.
\end{rmk}

\begin{rmk} Theorem \ref{theorem_main} is known to hold for certain projections of the process $Y$. Namely, in the
literature one can find proofs of the convergence towards the joint distribution of $W^N_i$,
$i=1,\dots,N$, the joint distribution of $W^i_i$, $i=1,\dots,N$, the joint distribution of
$W^i_1$, $i=1,\dots,N$ and the distribution of $W$ at any \emph{fixed} moment of time.
\end{rmk}

\medskip

Using the well-known convergence of the standard Poisson process to the standard Brownian motion,
we conclude from Theorem \ref{theorem_main} that the laws of the processes $(Y(nt)-tn)/\sqrt{n}$,
$t\geq0$ converge in the limit $n\to\infty$ to the law of the process $W$, started from $0$.
Similarly, applying Theorem \ref{theorem_main} to the interlacing dynamics of $N(N+1)/2$
independent simple random walks and using the convergence of the latter to standard Brownian
motions, we prove the conjecture in \cite{N} on the convergence of the particle dynamics related
to the shuffling algorithm for the domino tilings of the Aztec diamond to the process $W$ started
from $0$.

\bigskip

The proof of Theorem \ref{theorem_main} is based on the following idea. The limit process $W$ can
be constructed inductively using Skorokhod reflection maps in time-dependent intervals recently
introduced in \cite{BKR}, applied to the independent Brownian motions driving the particles.
Moreover, the prelimit processes can be seen to be the images of the respective driving processes
$X(\cdot;n)$ under similar Skorokhod maps. Putting these facts together with the joint continuity
of the Skorokhod map in the driving path and the time-dependent boundaries, one obtains Theorem
\ref{theorem_main}.

\smallskip

There are two directions in which Theorem \ref{theorem_main} might be generalized. First, it seems
plausible that the condition on the prelimit processes to have unit steps in Theorem
\ref{theorem_main} can be weakened. There are examples of the dynamics on Gelfand-Tsetlin patterns
in the literature (for example, interlacing dynamics driven by sums of geometric random variables
as in \cite{BF},\cite{WW}, or the shuffling algorithm for boxed plane partitions as in \cite{BG},
\cite{BGR}, \cite{Be}), for which the fixed time distribution is known to converge to the
``GUE-minors process''. Therefore, it is reasonable to expect that the laws of the paths of these
discrete processes also converge to the law of $W$. One might also want to study what happens when
the different components of the driving process have different asymptotic behavior, that is, when
the sequences $\{a_n(\cdot)\}_{n\in\nn}$, $\{b_n\}_{n\in\nn}$ in Theorem \ref{theorem_main} may also
depend on the particles. Some particular results into this direction were obtained in Section 7 of \cite{BG2}.

\smallskip

In addition, Theorem \ref{theorem_main} allows to recover many properties of the process $W$ (for
example, the ones studied in \cite{W}) from the corresponding properties of the discrete prelimit
processes, which are sometimes easier to prove (see e.g. \cite{BF}). For some more details in this
direction, please see Remark \ref{propofW} below.

\smallskip

It is worth noting that there is another natural way to construct a stochastic dynamics on
$\mathbb{GT}^{(N)}$ which is based on the Robinson-Schensted-Knuth correspondence, see \cite{OC1},
\cite{OC2}. This dynamics is driven by $N$ random walks as opposed to $N(N+1)/2$ in our case.
While the evolution itself is very much different from the one considered in the present article
due to strong correlations between components, the fixed time distributions (and also some
other projections) are similar and can be made the same by an appropriate choice of the driving
processes.

\smallskip

The rest of the article is organized as follows. In Section \ref{Section_definitions} we introduce
stochastic processes studied throughout the paper, in Section \ref{Section_Skorohod} we explain
the relation of  these processes to the Skorohod reflection map, and in Section
\ref{Section_Limit} we prove Theorem \ref{theorem_main}.

{\bf Acknowledgements.} The work on this article started while the authors were staying at MSRI
and we would like to thank the organizers of the program ``Random Spatial Processes''. We are also
grateful to A.~Borodin and I.~Corwin for useful discussions. V.G.\, was partially supported by
RFBR-CNRS grants 10-01-93114 and 11-01-93105.

\section{Processes on interlacing particle configurations}
\label{Section_definitions}

In this section, we give rigorous definitions of the stochastic dynamics on interlacing particle configurations that we study. To this end, let $X(t):=\{X_i^j(t),\;j=1,\dots,N,\;i=1,\dots,j\}$, $t\geq0$ be a random process taking values in $\mathbb{Z}^{N(N+1)/2}$. We impose the following \emph{regularity} conditions on $X$:
\begin{enumerate}
\item The trajectories of $X$ are c\`{a}dl\`{a}g, that is, for every $t\geq0$, the limits $X(t-):=\lim_{s\uparrow t} X(s)$ and $X(t+):=\lim_{s\downarrow t} X(s)$ exist, and $X(t)=X(t+)$.
\item Each coordinate $X_i^j$ has unit increments, that is, for every $t\geq0$, we have $|X_i^j(t)-X_i^j(t-)|\le 1$.
\end{enumerate}
Note that the regularity conditions imply, in particular, that on any time interval $[0,T]$, the trajectory of $X$ has finitely many points of discontinuity.

\bigskip

Given a regular process $X$, we construct the \emph{interlacing dynamics} $\mathcal X(t)$, $t\ge 0$ driven by $X$ and taking values in $\mathbb{GT}^{(N)}$ as follows. The initial value $\mathcal X(0)$ can be any $\mathbb{GT}^{(N)}$-valued random variable. If $X$ is constant on some time interval $[t_1,t_2]$, then so is $\mathcal X$. If a $t\geq0$ is a point of discontinuity in $X$, then the value of $\mathcal X(t)$ depends only on
$\mathcal X(t-)$ and $X(t)-X(t-)$, and is given by the following \emph{sequential update}. First, we define $\mathcal X_1^1(t)$, then $\mathcal X_i^2(t)$, $i=1,2$, then $\mathcal X_i^3(t)$, $i=1,2,3$, etc. To start with, we set
$$
\mathcal X_1^1(t)=\mathcal X_1^1(t-)+ X_1^1(t)-X^1_1(t-);
$$
in other words, the increments of the process $\mathcal X_1^1$ coincide with those of $X_1^1$. Subsequently, once the values of $\mathcal X^j_i(t)$, $j=1,\dots,k-1$, $i=1,\dots,j$ are determined, we define $\mathcal X_i^k(t)$ for each $i=1,\dots,k$ independently by the following procedure, in which each step is performed only if the conditions of the previous steps are not satisfied.
\begin{enumerate}
\item If $i>1$ and $\mathcal X_i^k(t-)=\mathcal X_{i-1}^{k-1}(t)-1$, then we say that particle $\mathcal X_i^k$ is \emph{pushed} by $\mathcal X_{i-1}^{k-1}$ and set $\mathcal X_{i}^{k}(t)=\mathcal X_i^k(t-)+1$.
\item If $i<k$ and $\mathcal X_i^k(t-)=\mathcal X_{i}^{k-1}(t)$, then we say that particle $\mathcal X_i^k$ is \emph{pushed} by $\mathcal X_{i}^{k-1}$ and set $\mathcal X_{i}^{k}(t)=\mathcal X_i^k(t-)-1$.
\item If $X_i^k(t)-X_i^k(t-)=1$, then we check whether $i<k$ and $\mathcal X_{i}^{k-1}(t)=\mathcal X_i^k(t-)+1$. If so, then we say that particle $\mathcal X_i^k$ is \emph{blocked} by $\mathcal X_i^{k-1}$ and set $\mathcal X_i^k(t)=\mathcal X_i^k(t-)$; otherwise, we set $\mathcal X_i^k(t)=\mathcal X_i^k(t-)+1$.
\item If $X_i^k(t)-X_i^k(t-)=-1$, then we check whether $i>1$ and $\mathcal X_{i-1}^{k-1}(t)=\mathcal X_i^k(t-)$. If so, then we say that particle $\mathcal X_i^k$ is \emph{blocked} by $\mathcal X_{i-1}^{k-1}$ and set $\mathcal X_i^k(t)=\mathcal X_i^k(t-)$; otherwise, we set $\mathcal X_i^k(t)=\mathcal X_i^k(t-)-1$.
\item If $X_i^k(t)-X_i^k(t-)=0$, then we set $\mathcal X_i^k(t)=\mathcal X_i^k(t-)$.
\end{enumerate}

\smallskip

The above definition can be also adapted to discrete-time processes. If $X(t)$, $t=0,1,2,\dots$ is such a process, then we define $\bar X(t)$, $t\ge 0$ by
\begin{equation}
\bar X(t)=X(n)\quad\text{for}\quad n\le t< n+1,\quad n=0,1,2,\dots
\end{equation}
and let $\mathcal X(t)$, $t=0,1,2,\dots$ be the restriction of the interlacing dynamics driven by $\bar X$ to integer times.

\bigskip

The following proposition enlists the properties of $\mathcal X$ which are immediate from the construction.

\begin{proposition}
$\mathcal X(t)$, $t\geq0$ is a well-defined stochastic process taking values in $\mathbb{GT}^{(N)}$. The paths of the process $\mathcal X$ depend only on the increments of the driving process $X$, but not on its initial value $X(0)$. Suppose that $X$ has independent increments, that is, for any $T>0$, the process $X(t)-X(T-)$, $t\ge T$ is independent of the process $X(t)-X(0)$, $0\leq t<T$. Then $\mathcal X$ is a Markov process.
\end{proposition}

\noindent{\bf Examples.}
\begin{enumerate}
\item If $X$ is a family of $N(N+1)/2$ independent standard Poisson processes, then $\mathcal X$ is the Markov process $Y$ defined in the introduction.
\item If $\{\xi_i^j(t):\;j=1,\dots,N,\;i=1,\dots,j,\;t=0,1,2\}$ is an array of i.i.d. Bernoulli random variables with parameter $p$ and $X_i^j(t)=\sum_{s=0}^{t-1} \xi_i^j(t)$, then $\mathcal X$ gives the dynamics studied in \cite{N} and \cite{BF}, which is related to the shuffling
algorithm for sampling random domino tilings of the Aztec diamond.
\end{enumerate}

\medskip

Next we aim to introduce a continuous-space analogue of the above discrete processes. This process $W$ was first defined and studied in \cite{W}. It takes values in the continuum version of $\mathbb{GT}^{(N)}$, which we denote by $\mathbb{GT}^{(N)}_c$. An element of $\mathbb{GT}^{(N)}_c$ is an array of reals $x_i^j$, $j=1,\dots,N$, $i=1,\dots,j$ satisfying the inequalities
\begin{equation}
x_{i-1}^j \le x_{i-1}^{j-1} \le x_{i}^{j}.
\end{equation}

The process $W$ taking values in $\mathbb{GT}^{(N)}_c$ is defined level by level. For $j=1$, we
let $W_1^1$ be a standard Brownian motion. For $j>1$, we define the component $W_i^j$ as a
standard Brownian motion (independent of those used in the previous steps) \emph{reflected} on the
trajectories of $W_{i-1}^{j-1}$ and $W_{i}^{j-1}$. In particular, $W_j^j$ and $W_1^j$ are
reflected on only one trajectory. We can start $W$ from any deterministic or random initial
condition with the only restriction being that $W(0)$ belongs to $\mathbb{GT}^{(N)}$ with
probability $1$. For more details, please see \cite[Section 2]{W} and \cite[Section 4]{W}.

\bigskip

\noindent{\bf Properties.} As shown in \cite{W}, the process $W$ started from $0$ possesses
certain interesting properties. In particular, the projection $W^N_i$, $i=1,\dots,N$ evolves
according to the Dyson's Brownian motion of dimension $N$. Also, for any fixed $t$, the
distribution of $W(t)$ is given by the distribution of the eigenvalues of corners of an $N\times
N$ Gaussian Hermitian matrix (see \cite{JN}, \cite{OR}).

\smallskip

Another intriguing property is the following Markovity: Fix an $N\in\nn$ and $0<t_1<t_2<\dots<t_N$
and consider the process $Z$ which is defined to coincide with the $\rr^k$-valued process
$(W^k_i:\;i=1,\dots,k)$ on the time interval $[t_{k-1},t_k)$ (here, we use the convention
$t_0=0$). Then the process $Z(t)$, $0\leq t\leq t_N$ is Markovian and its transitional
probabilities can be given explicitly, see \cite[proof of Proposition 6]{W} and also
\cite[Proposition 2.5]{BF} for a similar statement in discrete settings.

\section{Skorokhod maps}

\label{Section_Skorohod}

The proof of our main result relies on the observation that the limiting process $W$ is given by
the image of the vector of driving Brownian motions, started in $W(0)$, under an appropriate Skorokhod reflection map,
whereas the prelimit processes $\mathcal X(\cdot;n)$, $n\in\nn$ are given by images of the driving
processes $X(\cdot;n)$, $n\in\nn$, started in $\mathcal X(0;n)$, $n\in\nn$, respectively,
under slightly modified Skorokhod maps. Our construction relies on \cite[Theorem 2.6]{BKR},
which we state in the following proposition for the convenience of the reader.
\begin{proposition} \label{Sk_prop}
Let $l$, $r$ be two right-continuous functions with left limits on $[0,\infty)$ taking values in
$[-\infty,\infty)$ and $(-\infty,\infty]$, respectively. Suppose further that $l(t)\leq r(t)$,
$t\geq0$. Then, for every right-continuous function $\psi$ with left limits on $[0,\infty)$ taking
values in $\rr$, there exists a unique pair of functions $\phi$, $\eta$ in the same space
satisfying
\begin{enumerate}[1.]
\item For every $t\in[0,\infty)$, $\phi(t)=\psi(t)+\eta(t)\in[l(t),r(t)]$;
\item For every $0\leq s\leq t<\infty$,
\begin{eqnarray}
&&\eta(t)-\eta(s)\geq0\quad \text{if}\;\phi(u)<r(u)\;\text{for\;all\;}u\in(s,t],\\
&&\eta(t)-\eta(s)\leq0\quad \text{if}\;\phi(u)>l(u)\;\text{for\;all\;}u\in(s,t];
\end{eqnarray}
\item For every $0\leq t<\infty$,
\begin{eqnarray}
\eta(t)-\eta(t-)\geq0\quad\text{if}\;\phi(t)<r(t),\\
\eta(t)-\eta(t-)\leq0\quad\text{if}\;\phi(t)>l(t),
\end{eqnarray}
where $\eta(0-)$ is to be interpreted as $0$.
\end{enumerate}
Moreover, the map $\Gamma$, which sends the triple $(l,r,\psi)$ to the function $\phi$ is jointly
continuous with respect to the topology of uniform convergence on compact sets.
\end{proposition}

We will refer to  $\Gamma(l,r,\cdot)$ as the \textit{extended Skorokhod map} in the time-dependent
interval $[l(t),r(t)]$, $t\geq0$. It will not be important in the following, but it is worth
mentioning that the map $\Gamma$ can be given explicitly (see \cite[Theorem 2.6]{BKR}).

\subsection{Skorokhod maps for the prelimit processes}

Since the definition of $\mathcal X(\cdot;n)$ does not depend on the initial value $X(0;n)$ we may
assume that $\mathcal X(0;n)=X(0;n)$. In this subsection we show that in this case each of the
prelimit processes $\mathcal X(\cdot;n)$, $n\in\nn$ can be constructed as the image of the vector
of driving processes under a deterministic map, which we will refer to as the discrete Skorokhod
map.

\bigskip

In our construction we fix  an $n\in\nn$ and proceed by induction over the number of levels $N$.
For $N=1$, we set $\tilde{\mathcal X}(\cdot;n)=X(\cdot;n)=\mathcal X(\cdot;n)$. For $N\geq2$, we
may assume that the process $\tilde{\mathcal X}(\cdot;n)$ with $(N-1)$ levels, with initial condition
being the restriction of $\mathcal X(0;n)$ to the first $(N-1)$ levels, has already been
constructed. We fix a path of this process and will construct the corresponding paths of the
particles on level $N$.

\bigskip

We consider first an $1<i<N$ and define $\tilde{\mathcal X}^N_i(\cdot;n)$ as the image of $X^N_i(\cdot;n)$, started in $\mathcal X^N_i(0;n)$, under the extended Skorokhod map in the time-dependent interval
\begin{equation}
\label{eq_timedep_1} \big[\tilde{\mathcal X}^{N-1}_{i-1}(t;n),\tilde{\mathcal
X}^{N-1}_i(t;n)-1\big],\quad t\geq0
\end{equation}
in the sense of Proposition \ref{Sk_prop} above. Next, let $i=1$. In this case, we define $\tilde{\mathcal X}^N_1(\cdot;n)$ as the image of $X^N_1(\cdot;n)$, started in $\mathcal X^N_1(0;n)$, under the extended Skorokhod map in the time-dependent interval
\begin{equation}
\label{eq_timedep_2} \big(-\infty,\tilde{\mathcal X}^{N-1}_1(t;n)-1\big], \quad t\geq0
\end{equation}
in the sense of Proposition \ref{Sk_prop}. Similarly, for $i=N$, we define $\tilde{\mathcal X}^N_N(\cdot;n)$ as the image of $X^N_N(\cdot;n)$, started in $\mathcal X^N_N(0;n)$, under the extended Skorokhod map in the time-dependent interval
\begin{equation}
\label{eq_timedep_3} \big[\tilde{\mathcal X}^{N-1}_{N-1}(t;n),\infty\big),\quad t\geq0
\end{equation}
in the sense of Proposition \ref{Sk_prop}. This finishes the construction.

\bigskip

We remark at this point that the paths of the processes $\tilde{\mathcal X}(\cdot;n)$, $n\in\nn$ are given by images of the paths of the driving processes $X(\cdot;n)$, $n\in\nn$, started in $\mathcal X(0;n)$, $n\in\nn$, respectively, under a \textit{deterministic} map $\Phi^{SK}$ depending only on $N$ (and not on $n$):
\begin{equation}
\tilde{\mathcal X}(\cdot;n)=\Phi^{SK}(X(\cdot;n)), \quad n\in\nn.
\end{equation}
We will refer to $\Phi^{SK}$ as the \textit{discrete Skorokhod map}. Finally, for each $n\in\nn$,
we define $\tilde{\mathcal Y}(\cdot;n)$ as the process constructed from
$\frac{X(\cdot;n)-a_n(\cdot)}{b_n}$ by employing the discrete Skorokhod map in the rescaled
coordinates; that is, repeating the procedure above, but rescaling all initial conditions,
processes involved and the constant $1$ in \eqref{eq_timedep_1}, \eqref{eq_timedep_2} according to the
change of coordinates
\begin{equation}\label{rescaled}
[0,\infty)\times\zz^{N(N+1)/2}\rightarrow[0,\infty)\times\rr^{N(N+1)/2},\quad (t,x)\mapsto\Big(t,\frac{x-a_n(t)}{b_n}\Big).
\end{equation}
The following proposition shows that $\tilde{\mathcal X}(\cdot;n)$, $n\in\nn$ are in fact the processes defined in
Theorem \ref{theorem_main}, and that $\tilde{\mathcal Y}(\cdot;n)$, $n\in\nn$ coincide with the
corresponding rescaled processes.

\begin{proposition}\label{propSK_disc}
Suppose that the processes $\tilde{\mathcal X}(\cdot;n)$, $n\in\nn$ and $\mathcal X(\cdot;n)$, $n\in\nn$ are constructed by using the same driving processes $X(\cdot;n)$, $n\in\nn$. Then,
\begin{eqnarray}
&&\tilde{\mathcal X}(\cdot;n)=\mathcal X(\cdot;n),\quad n\in\nn, \label{match1}\\
&&\tilde{\mathcal Y}(\cdot;n)=\frac{\mathcal X(\cdot;n)-a_n(\cdot)}{b_n},\quad n\in\nn \label{match2}
\end{eqnarray}
with probability $1$.
\end{proposition}

\begin{proof}
Fix an $n\in\nn$. We start with the proof of \eqref{match1} for that $n$ and proceed by induction
over the number of levels $N$. For $N=1$, the statement \eqref{match1} follows directly from the
definition of $\tilde{\mathcal X}(\cdot;n)$. Now, let $N\geq2$ and suppose that \eqref{match1}
holds for all $N'<N$. Then, the paths of the particles on the first $(N-1)$ levels of
$\tilde{\mathcal X}(\cdot;n)$ must coincide with the paths of the particles on the first $(N-1)$
levels of $\mathcal X(\cdot;n)$. Now, consider the trajectories of the particles on the $N$-th
level in the two processes. If the trajectories coincide up to some time $t_1$, then they clearly
coincide up to time $t_2>t_1$ as long as no particles are pushed or blocked in $\mathcal
X(\cdot;n)$ in the time interval $[t_1,t_2]$.

\bigskip

In case that a particle on level $N$ in $\mathcal X(\cdot;n)$ is pushed, the trajectory of the
pushing particle must have a jump of size $1$ at the moment of the push and the other particle
is pushed accordingly. The definitions of $\Phi^{SK}$ and of the extended Skorokhod map in a time-dependent
interval incorporated in $\Phi^{SK}$ show that the value of $\tilde{\mathcal X}(\cdot;n)=\Phi^{SK}(X(\cdot;n))$ after the
push must coincide with the value of $\mathcal X(\cdot;n)$ after the push.

\bigskip

Similarly, when a particle on level $N$ in $\mathcal X(\cdot;n)$ is blocked, the blocking particle
must be at distance $1$ from the blocked particle at the moment, when it blocks the other particle,
and the trajectory of the component in $X(\cdot;n)$ driving the blocked particle must have a jump
of size $1$ at the same moment. By the definitions of $\Phi^{SK}$ and the extended Skorokhod map in a time-dependent
interval, the value of $\tilde{\mathcal X}(\cdot;n)=\Phi^{SK}(X(\cdot;n))$ after a particle is
blocked must coincide with the value of $\mathcal X(\cdot;n)$ at the same moment of time.

\bigskip

The statement \eqref{match2} can be shown by arguing as in the proof of the statement
\eqref{match1}, but working in the coordinates obtained by the rescaling \eqref{rescaled}.
\end{proof}

\subsection{Skorokhod map for the limiting process}

In this subsection we will employ Proposition \ref{Sk_prop} to show that the process $W$ can be viewed as the image of the vector of driving Brownian motions, started in $W(0)$, under a deterministic map, which we will refer to as the \textit{continuum Skorokhod map}. To this end, we let $B:=\{B^j_i,\,j=1,\dots,N,\,i=1,\dots,j\}$ be the collection of independent standard Brownian motions driving the process $W$ and will construct an auxilliary $\mathbb{GT}^{(N)}_c$-valued process $\tilde{W}$ on the same probability space, which \textit{a posteriori} will turn out to coincide with $W$.

\bigskip

The construction proceeds by induction over the number of levels $N$. For $N=1$, we set $\tilde{W}=B^1_1=W$. For $N\geq2$, we may assume that the paths of the particles on the first $(N-1)$ levels of the process $\tilde{W}$, with initial condition being the restriction of $W(0)$ to the first $(N-1)$ levels, have already been constructed. Then, for $1<i<N$, we define the path of the particle $\tilde{W}^N_i$ as the image of the driving Brownian motion $B^N_i$, started in $W^N_i(0)$, under the extended Skorokhod map in the time-dependent interval
\begin{equation}
\big[\tilde{W}^{N-1}_{i-1}(t),\tilde{W}^{N-1}_i(t)\big],\quad t\geq0
\end{equation}
in the sense of Proposition \ref{Sk_prop}. For $i=1$, we define the path of the particle $\tilde{W}^N_1$ as the image of the driving Brownian motion $B^N_1$, started in $W^N_1(0)$, under the extended Skorokhod map in the time-dependent interval
\begin{equation}
\big(-\infty,\tilde{W}^{N-1}_1(t)\big],\quad t\geq0
\end{equation}
defined as in Proposition \ref{Sk_prop}. Finally, for $i=N$, we define the path of the particle $\tilde{W}^N_N$ as the image of the driving Brownian motion $B^N_N$, started in $W^N_N(0)$, under the extended Skorokhod map in the time-dependent interval
\begin{equation}
\big[\tilde{W}^{N-1}_{N-1}(t),\infty\big),\quad t\geq0
\end{equation}
in the sense of Proposition \ref{Sk_prop}.

\bigskip

We note that the process $\tilde{W}$ is given by the image of the Brownian motion $B$, started in $W(0)$, under a \textit{deterministic} map $\Psi^{SK}$:
\begin{equation}
\tilde{W}=\Psi^{SK}(B).
\end{equation}
We will refer to $\Psi^{SK}$ as the \textit{continuum Skorokhod map}. The following proposition
shows that $\tilde{W}$ coincides with $W$.

\begin{proposition}\label{SKprop_cont}
Let the processes $\tilde{W}$ and $W$ be defined on the same probability space and be driven by the same collection $B$ of independent standard Brownian motions. Then, with probability $1$:
\begin{equation}
\tilde{W}(t)=W(t)\quad\text{for\;all}\quad t\geq0.
\end{equation}
\end{proposition}

\begin{proof}
We proceed by induction over the number of levels $N$. For $N=1$, the statement of the proposition
follows directly from the definition of $\tilde{W}$. Now, consider an $N\geq2$ and assume that the
proposition holds for all $N'<N$. Since the paths of the particles on the first $(N-1)$ levels in
$\tilde{W}$ and $W$ coincide with the paths of the particles in the respective processes with
$(N-1)$ levels, we have $\tilde{W}^j_i=W^j_i$ for all $j=1,\dots,N-1$, $i=1,\dots,j$ with
probability $1$.

\bigskip

Now, fix an $1<i<N$ and consider the particle $\tilde{W}^N_i$ on the $N$-th level in $\tilde{W}$ and
the corresponding particle $W^N_i$ in $W$. Since both $\tilde{W}^N_i$ and $W^N_i$ solve the extended
Skorokhod problem in the time-dependent interval
\begin{equation}
\big[\tilde{W}^{N-1}_{i-1}(t),\tilde{W}^{N-1}_i(t)\big]=\big[W^{N-1}_{i-1}(t),W^{N-1}_i(t)\big],\quad t\geq0
\end{equation}
for the Brownian motion $B^N_i$, started in $W^N_i(0)$, by their respective constructions, the uniqueness statement of
Proposition \ref{Sk_prop} implies that it must hold $\tilde{W}^N_i=W^N_i$ with probability $1$. In
the cases that $i=1$ or $i=N$, one can argue in the same manner, invoking the uniqueness statement
of  Proposition \ref{Sk_prop} for the extended Skorokhod maps in the time-dependent intervals
\begin{eqnarray}
&&\big(-\infty,\tilde{W}^{N-1}_1(t)\big]=\big(-\infty,W^{N-1}_1(t)\big],\quad t\geq0,\\
&&\big[\tilde{W}^{N-1}_{N-1}(t),\infty\big)=\big[W^{N-1}_{N-1}(t),\infty\big),\quad t\geq0,
\end{eqnarray}
respectively, to conclude that $\tilde{W}^N_1=W^N_1$ and $\tilde{W}^N_N=W^N_N$ almost surely. This finishes the proof of the proposition.
\end{proof}

\begin{rmk}
Although this will not play a role in the following, we remark that due to Proposition
\ref{SKprop_cont} and the findings in \cite{W} the particles on the same level in $\tilde{W}$
never collide after time $0$. Thus, if at time $0$ no two of these particles are at the same
location, \cite[Corollary 2.4]{BKR} shows that the extended Skorokhod maps used in the
construction of $\tilde{W}$ are in fact regular Skorokhod maps in the sense of \cite[Definition
2.1]{BKR}.
\end{rmk}

\section{Limiting procedure}

\label{Section_Limit}

We now come to the proof of the main result of the paper, stating that the process $W$ serves as a universal limiting object for discrete processes with block/push interactions.

\bigskip

\noindent\textit{Proof of Theorem \ref{theorem_main}.} We proceed by induction over the number of levels $N$. For $N=1$, there is nothing to show. We now let $N\geq2$ be fixed and assume that the theorem holds for all $N'<N$. Since the first $(N-1)$ levels of the processes $\frac{\mathcal X(\cdot;n)-a_n(\cdot)}{b_n}$, $n\in\nn$ evolve as the corresponding processes with $(N-1)$ levels, we conclude that their laws converge to the law of $\{W^j_i:\;j=1,\dots,N-1,\;i=1,\dots,j\}$ on $D([0,\infty),\rr^{(N-1)N/2})$, started according to the restriction of $W(0)$ to the first $(N-1)$ levels. By the Skorokhod Embedding Theorem in the form of \cite[Theorem 3.5.1]{Du}, one can define the processes on the first $(N-1)$ levels of $\mathcal X(\cdot;n)$ together with the processes $X^N_i(\cdot;n)$, $i=1,\dots,N$, started in $\mathcal X^N_i(0;n)$, $i=1,\dots,N$, respectively, for all values of $n\in\nn$ on the same probability space such that
\begin{eqnarray}
&&\frac{\mathcal X^j_i(\cdot;n)-a_n(\cdot)}{b_n}\xrightarrow{\smash{n\rightarrow\infty}} W^j_i,\quad j=1,\dots,N-1,\;i=1,\dots,j,\\
&&\frac{X^N_i(\cdot;N)-a_n(\cdot)}{b_n} \xrightarrow{\smash{n\rightarrow\infty}} B^N_i,\quad i=1,\dots,N
\end{eqnarray}
almost surely, where $\{W^j_i:\,j=1,\dots,N-1,\,i=1,\dots,j\}$ is a $\mathbb{GT}^{(N-1)}_c$-valued process of the law described in Theorem \ref{theorem_main} and $(B^N_1,\dots,B^N_N)$ is an $N$-dimensional standard Brownian motion started in $(W^N_1(0),\dots,W^N_N(0))$.

\bigskip

Now, recall from Proposition \ref{propSK_disc} and the definition of the processes $\tilde{\mathcal Y}(\cdot;n)$, $n\in\nn$ that the paths of the processes $\frac{\mathcal X^N_i(\cdot;n)-a_n(\cdot)}{b_n}$, $i=1,\dots,N$, $n\in\nn$ are given by the images of the driving processes $\frac{X^N_i(\cdot;n)-a_n(\cdot)}{b_n}$, $i=1,\dots,N$, $n\in\nn$ under the extended Skorokhod maps in the appropriate time-dependent intervals (we will refer to the latter as $I^N_i(n)$). Moreover, from Proposition \ref{SKprop_cont} and the definition of the process $\tilde{W}$ we know that the paths of the processes $W^N_i$, $i=1,\dots,N$ are given by the images of the driving Brownian motions $B^N_1,\dots,B^N_N$ under the extended Skorokhod maps in the appropriate time-dependent intervals (we will refer to the latter as $I^N_i(\infty)$).

\bigskip

It now suffices to observe that, for all $1\leq i\leq N$, the left (resp. right) boundaries of the time-dependent intervals $I^N_i(n)$ converge in the limit as $n\to\infty$ to the left (resp. right) boundaries of the time-dependent intervals $I^N_i(\infty)$ in $D([0,\infty),\rr)$ with probability $1$. We remark that here the assumption $b_n\to\infty$ as $n\to\infty$ is used. By applying the continuity statement in Proposition \ref{Sk_prop} we conclude that the paths of the processes $\frac{\mathcal X^N_i(\cdot;n)-a_n(\cdot)}{b_n}$, $i=1,\dots,N$ converge in the limit as $n\to\infty$ to the paths of the processes $W^N_i$, $i=1,\dots,N$ almost surely in $D([0,\infty),\rr)$. This finishes the proof. \ep

\bigskip

\begin{rmk}
The proof of Theorem \ref{theorem_main}  shows that the result of the theorem is universal in the
following sense. Suppose that $X(\cdot;n)$, $n\in\nn$ are such that the processes
$\frac{X(\cdot;n)-a_n(\cdot)}{b_n}$, $n\in\nn$ converge in law to an arbitrary
$D([0,\infty),\rr^{N(N+1)/2})$-valued process $B$ as $n\to\infty$, where $\{a_n(\cdot)\}_{n\in\nn}$ is
a sequence of continuous real-valued functions on $[0,\infty)$ and
$\{b_n\}_{n\in\nn}$ is a sequence of positive reals such that $b_n\to\infty$ as $n\to\infty$. Then, the
laws of the processes $\frac{\mathcal X(\cdot;n)-a_n(\cdot)}{b_n}$, $n\in\nn$ converge to the law of the
image of $B$ under the continuum Skorokhod map $\Psi^{SK}$ in $D([0,\infty),\rr^{N(N+1)/2})$ as
$n\to\infty$.
\end{rmk}

\smallskip

\begin{rmk}\label{propofW}
Theorem \ref{theorem_main} shows that one can reprove the properties of the process $W$ and its
projections stated at the end of Section \ref{Section_definitions} by using the corresponding
properties of the process $Y$ and taking limits. Indeed, one can reexpress the transition
probabilities of the process $W$ as limits of the corresponding transition probabilities of the
appropriately rescaled versions of the process $Y$.
\end{rmk}

\end{document}